\documentclass[a4paper,reqno]{amsart}

\usepackage{amsmath}
\usepackage{amsthm}
\usepackage{amssymb}
\usepackage{eucal}
\usepackage{tikz}
\usepackage[backref,colorlinks]{hyperref}

\numberwithin{equation}{section}
\usetikzlibrary{cd,arrows}
\tikzset{vertex/.style={shape=circle,draw,minimum size=3pt,inner sep=0pt}}
\tikzset{black-vertex/.style={shape=circle,draw,inner sep=3pt}}
\tikzset{red-vertex/.style={shape=circle,draw=red,dashed,inner sep=3pt}}
\tikzset{edge/.style={->,>=latex'}}

\theoremstyle{plain}
\newtheorem{mainthm}{Theorem}

\newtheorem{thm}{Theorem}[section]
\newtheorem{prop}[thm]{Proposition}
\newtheorem{cor}[thm]{Corollary}
\newtheorem{lem}[thm]{Lemma}
\theoremstyle{definition}

\newtheorem{exm}[thm]{Example}
\theoremstyle{remark}
\newtheorem{rmk}[thm]{Remark}

\newcommand{\bbA}{\mathbb{A}}

\DeclareMathOperator{\End}{End}

\DeclareMathOperator{\soc}{soc}
\renewcommand{\mod}{\operatorname{mod}}
\DeclareMathOperator{\defect}{def}
\DeclareMathOperator{\dist}{dist}
\DeclareMathOperator{\cyc}{cyc}

\newcommand{\scB}{\mathcal{B}}

\newcommand{\scF}{\mathcal{F}}

\newcommand{\scP}{\mathcal{P}}

\begin{document}

\title[Resolution quiver of Nakayama algebras]
      {The resolution quiver of Nakayama algebras which are minimal Auslander-Gorenstein}
\author{Dawei Shen}
\address{School of mathematics and statistics \\
         Henan University \\
         Kaifeng, Henan 475004 \\
         P. R. China}
\email{sdw12345@mail.ustc.edu.cn}
\subjclass[2020]{Primary 16E10; Secondary 16G20}
\keywords{Nakayama algebra, Auslander-Gorenstein algebra, resolution quiver}
%\dedicatory{}
%\commby{}
\date{\today}

\begin{abstract}
Let $A$ be a Nakayama algebra. 
Using Ringel's resolution quiver, 
we give a criterion to determine 
whether $A$ is a minimal Auslander-Gorenstein algebra.
The criterion  strongly relies on the parity of the selfinjective dimension of $A$.
\end{abstract}
\maketitle

\section{Introduction}

Following Iyama \cite{Iya07}, an algebra is said to be higher Auslander 
provided that its global dimension is finite and bounded by its dominant dimension.
Higher Auslander algebras generalize the classical Auslander algebras.
Following Iyama and Solberg \cite{IS18}, 
an algebra is said to be minimal Auslander-Gorenstein
provided that its selfinjective dimension is finite and bounded by its dominant dimension.
Minimal Auslander-Gorenstein algebras are the Gorenstein analogues of higher Auslander algebras.
They are a special class of Auslander-Gorenstein algebras.

Nakayama algebras are one of the most fundamental classes of algebras 
in the representation theory of algebras.
An algebra is called a Nakayama algebra 
if it is a homomorphic image of the path algebra of an oriented cycle.
The indecomposable modules over Nakayama algebras are uniserial.
A connected Nakayama algebra is either linear or cyclic.

The resolution quiver for  Nakayama algebras is introduced by Ringel \cite{Rin13}.
It is an efficient tool for studying the homological properties of Nakayama algebras. 
A resolution quiver criterion to detect the finiteness of global dimension 
and the finiteness of  selfinjective dimension for Nakayama algebras is given in \cite{She17}.

The syzygy filtered algebra for cyclic Nakayama algebras is introduced by Sen \cite{Sen19}.
It plays a role in constructing cyclic Nakayama algebras which are higher Auslander algebras \cite{Sen20}.
The reversed syzygy filtered algebra for all Nakayama algebras is introduced in \cite{STZ24}.
It provides a systematic method to construct cyclic Nakayama algebras which are 
higher Auslander and minimal Auslander-Gorenstein.
A classification of defect one cyclic Nakayama algebras is given in \cite{MMG20},
and a classification of  concave linear Nakayama algebras which are higher Auslander is given in \cite{Rin22}.

Inspired by their work, we study the syzygy filtered algebra for all Nakayama algebras.
Then we give a resolution quiver criterion to determine whether a Nakayama algebra 
is a minimal Auslander-Gorenstein algebra.
This criterion  strongly  depends on the parity of the selfinjective dimension.

Let $A$ be a  Nakayama algebra with translation $\tau$ on the simple $A$-modules.
For an $A$-module $M$, denote by $\ell(M)$ the composition length, 
by $P(M)$ the projective cover, and by $I(M)$ the injective envelope of $M$.
The resolution quiver $R(A)$ of $A$ is defined as follows. 
The vertices of $R(A)$ are the isoclasses of all simple $A$-modules, 
and there is a unique arrow from $S$ to $\tau\soc P(S)$ for each  simple $A$-module $S$.
A simple $A$-module $S$ is said to be red if the projective dimension of $S$ is equal to $1$, 
and black otherwise.
This  assigns a color to each vertex of the resolution quiver.

A simple $A$-module $S$ is said to be cyclic if $S$ lies on a cycle in $R(A)$, 
and a leaf if there is no arrow in $R(A)$ ending at $S$.
For any simple $A$-module $S$, we denote by $\dist S$ the minimal integer $d$ 
such that the $d$-th successor of $S$  is cyclic.
Let $\dist A$ be the maximum of $\dist S$ over all simple $A$-modules $S$.

The weight  of $R(A)$ is the sum of  $\ell(P(S))/nc$, 
where $n$ is the rank of $A$, $c$ is the number of cycles in $R(A)$, 
and $S$ runs through all cyclic vertices in $R(A)$.
It turns out that a Nakayama algebra has finite global dimension if and only if 
its resolution quiver is connected of weight $1$.

The following is the main result of this paper.

\begin{mainthm} \label{mainthm:A}
    Let $A$ be a Nakayama  algebra and  $R(A)$  its resolution quiver. 
    Then $A$ is minimal Auslander-Gorenstein of {\bfseries odd} selfinjective dimension if and only if 
    \begin{enumerate}
        \item $\dist S=\dist A>0$ for every leaf $S$ in $R(A)$;
        \item every noncyclic vertex has at most $1$ predecessor in $R(A)$;
        \item $R(A)$ is connected of weight $1$; 
        \item $\tau^{-1}$ of all cyclic vertices in $R(A)$ are black.
    \end{enumerate}
\end{mainthm}

\begin{mainthm} \label{mainthm:B}
    Let $A$ be a  Nakayama algebra and  $R(A)$ its resolution quiver. 
    Then $A$ is minimal Auslander-Gorenstein of {\bfseries even} selfinjective dimension if and only if 
    \begin{enumerate}
        \item $\dist S=\dist A>0$ for every leaf $S$ in $R(A)$;
        \item every noncyclic vertex has at most $1$ predecessor in $R(A)$;
        \item every cyclic vertex has at most $2$ predecessors in $R(A)$;
        \item all cyclic vertices in $R(A)$ are black.
    \end{enumerate}
\end{mainthm}

Throughout this paper, $k$ is a fixed algebraically closed field.
All algebras are finite dimensional associative algebras over $k$.
All modules are finite dimensional left modules unless otherwise stated.
We refer to \cite{ARS95} for basics on finite dimensional algebras and their representations.

The  paper is organized as follows.
In Section 2, we recall some basic properties of Nakayama algebras and their resolution quivers.
In Section 3, we investigate the syzygy filtered algebra for all Nakayama algebras.
Then we prove the main result in Section 4.
Some applications and examples are  given in Section 5.

\subsection*{Acknowledgements}
The author is grateful to Professor Shijie Zhu for his comments and suggestions.
This work is supported by the National Natural Science Foundation of China 
(No. 12571036 and 11801141).

\section{Preliminary}

Given an integer  $n$, we denote by $\Delta_n$ the quiver on $n$ vertices containing a single 
directed cycle through all vertices; see Figure \ref{fig:cycle}.
Denote by $k\Delta_n$  the path algebra of  $\Delta_n$.

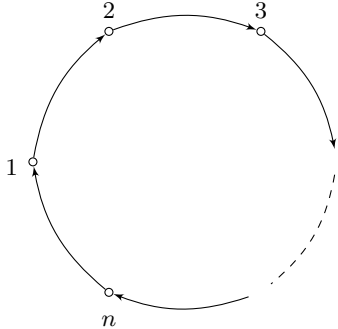
\begin{figure}
\begin{tikzpicture}
    % vertices
\node[vertex, label={[xshift=-0.8em, yshift=-1em] \small $1$}] (1) at  (-2,0) {};
\node[vertex, label={\small $2$}] (2) at  (-1,1.73) {};
\node[vertex, label={\small $3$}] (3) at  (1,1.73) {};
\node[shape=circle] (4) at (2,0){};
\node[shape=circle] (5) at (1,-1.73){};
\node[vertex, label={[yshift=-1.8em]\small $n$}] (6) at (-1,-1.73) {};
%edges
\draw[edge] (1) to[bend left=20] (2);
\draw[edge] (2) to[bend left=20] (3);
\draw[edge] (3) to[bend left=20] (4);
\draw[dashed] (4) to[bend left=20] (5) ;
\draw[edge] (5) to[bend left=20] (6);
\draw[edge] (6) to[bend left=20] (1);
\end{tikzpicture}
    \caption{The quiver $\Delta_n$.}
    \label{fig:cycle}
\end{figure}

Following \cite{HI20}, an  algebra $A$ is said to be 
a Nakayama algebra of rank $n$ if 
it is isomorphic to the quotient algebra of $k\Delta_n$ 
via the ideal generated by a finite set $I$ of nontrivial  paths.
If $I$ consists of paths of length  at least $2$,
then $A$ is a cyclic Nakayama algebra.
If there exists a unique arrow in $I$,
then $A$ is a linear Nakayama algebra.
Both cyclic and linear Nakayama algebras are connected.
We allow more than one arrow in $I$.

Let $A$ be a Nakayama algebra of rank $n$.
For each vertex $a$ in $\Delta_n$, denote by $S(a)$ the simple $A$-module, 
by $P(a)$ the projective $A$-module, and by $I(a)$ the injective $A$-module at $a$.
Set $\tau(a) = b$ if there is an arrow from  $a$ to $b$ in $\Delta_n$.
This induces a translation $\tau$ on the simple $A$-modules. 
If $A$ is a cyclic Nakayama algebra, 
then $\tau$ coincides with the Auslander-Reiten translation \cite{ARS95}.

For each vertex $a$ in $\Delta_n$, 
denote by $c_a$  the composition length of $P(a)$.
The series 
\[c(A)=(c_1,c_{2},\cdots, c_{n})\] 
is called a Kupisch series of $A$.
Two Nakayama algebras are isomorphic if and only if 
their Kupisch series are equal up to cyclic permutation.

Following \cite{Rin13}, the resolution quiver $R(A)$ of $A$ 
is the successor graph of the function 
\[\gamma(S)=\tau\soc P(S)\]
on the set of all isoclasses of simple $A$-modules; see also \cite{Gus85}.
A simple $A$-module $S$ is said to be red if the 
projective dimension of $S$ is  equal to $1$, and black otherwise.
A projective $A$-module $P$ is said to be minimal if 
its top module is a direct sum of black simple $A$-modules.

The following is well known.

\begin{lem} \label{lem:black}
Let $S$ be a simple $A$-module. Then
\begin{enumerate}
   \item  $P(S)$ is injective if and only if  $\tau^{-1}S$ is black;
   \item  $\ell(P(S))=\ell(P(\tau S))+1$ if and only if $S$ is red.
\end{enumerate}
\end{lem}

\begin{cor} \label{cor:two-black}
Let $S$ be a simple $A$-module. Then $S$ and $\tau^{-1}S$ are black 
if and only if $\gamma^{-1}(\gamma S)=\{S\}$.
\end{cor}

We need the following; see \cite[1.1]{She17} and \cite[3.10]{Rin21}.

\begin{lem}\label{lem:finite-gld}
    Let $A$ be a Nakayama  algebra. Then
    \begin{enumerate} 
    \item $A$ has finite global dimension if and only if $R(A)$ is connected of weight $1$;
    \item $2\dist A-1\leq \mathrm{gl}.\dim A\leq 2\dist A$, provided that $\mathrm{gl}.\dim A$ is finite.
    \end{enumerate}
\end{lem}

An algebra $A$ is said to be Gorenstein if its left selfinjective dimension 
and right selfinjective dimension  are finite.
If $A$ is a Gorenstein algebra, then its left and right selfinjective dimensions 
are equal; see \cite{Zak69}.   
It is known that a Nakayama algebra is Gorenstein if and only if 
its left selfinjective dimension is finite.

We need the following; see \cite[1.2]{She17} and \cite[4.13]{Sen19}.

\begin{lem}\label{lem:finite-vd}
Let $A$ be a Nakayama algebra of  infinite global dimension. Then
    \begin{enumerate}
        \item $A$ is Gorenstein if and only if all cyclic vertices in $R(A)$ are black;
        \item $\mathrm{id}_A\,A=2\dist A$, provided that $\mathrm{id}_A\,A$ is finite. 
    \end{enumerate} 
\end{lem}

Let $A$ be a Nakayama algebra.
A simple $A$-module $S$ is said to be a leaf if there is no arrow in $R(A)$ ending at $S$.
We need the following; see \cite[3.1]{Mad05}.

\begin{lem} \label{lem:leaf}
Let $S$ be a simple $A$-module. Then
\begin{enumerate}
   \item  $I(S)$ is projective if and only if $\tau S$ is a non-leaf in $R(A)$;
   \item  $\mathrm{id}_A\,S=1$  if and only if $S$ is a leaf in $R(A)$.
\end{enumerate}
\end{lem}

\section{Syzygy filtered algebra}
In this section, we study the syzygy filtered algebra for all Nakayama algebras.
This generalizes  Sen's construction for cyclic Nakayama algebras.

Let $A$ be a Nakayama algebra and $R(A)$ its resolution quiver.
By \cite[A.6]{Rin21} the map $\gamma$  induces a one-to-one correspondence between 
the set of all black vertices and the set of all non-leaf vertices in $R(A)$.
For any non-leaf vertex $S$ in $R(A)$, let $T$ be the unique black vertex such that $\gamma(T)=S$.
Consider the  exact sequence 
\begin{equation}\label{eq:elementary}
0 \to \Delta(S) \to  P(\tau T)\overset\alpha\to  P(T) \to  T\to 0,
\end{equation}
where $\alpha\in A$ is the arrow starting at  $T$.
The kernel $\Delta(S)$ is called a base module. 
It  is the quotient module of $P(S)$ of minimal length
such that the translation of its socle is a non-leaf vertex.

For any module $M$, denote by $\Omega(M)$ the syzygy module of $M$.
Then we have
\begin{equation} \label{eq:delta}
\Delta(S)=\begin{cases}
    P(S), &\text{if}\;\mathrm{pd}_A\,T\in\{0,2\},\\
   \Omega^2(T),&\text{otherwise}.
\end{cases}
\end{equation}
Following \cite{Sen19,Sen21,Rin21}, the base set 
\[\scB=\{\Delta(S)\mid \text{$S$ is a non-leaf vertex in $R(A)$}\}\] 
form a semibrick in the category $\mod A$ of all finite dimensional $A$-modules. 
The filtration category $\scF$ of $\scB$ is an exact Abelian subcategory of $\mod A$. 
The category $\scF$ contains the projective cover  of every base module 
and the second syzygy of every finite dimensional $A$-module.

Denote by $\scP$ the projective covers of all base modules.
The syzygy filtered algebra $\varepsilon(A)$ of $A$ is the endomorphism algebra  of $\scP$, that is,
\[\varepsilon(A)=\End_A(\bigoplus\nolimits_{P\in \scP}P).\]
It is a Nakayama algebra of rank less than or equal to the rank of $A$. 
The equality holds if and only if $A$ is selfinjective.
The category  $\mod \varepsilon(A)$  of all finite dimensional 
$\varepsilon(A)$-modules is equivalent to  $\scF$. 
The syzygy filtered algebra $\varepsilon(A)$ has finite global dimension 
if and only if $A$ has finite global dimension.

Given  a Nakayama algebra $A$, denote by $\chi(A)$ the number of isoclasses of simple projective $A$-modules.
Note that $\chi(A)$ equals the number of occurrences of $1$ in the Kupisch series of $A$.
For any integer $i\geq 0$, 
denote by $\varepsilon^i(A)$ the $i$-th iterated syzygy filtered algebra of $A$.
By \eqref{eq:delta}, we have the following.

\begin{lem}
Let $A$ be a Nakayama algebra and $i\geq 1$. 
Then $\chi(\varepsilon^i{A})-\chi(\varepsilon^{i-1}A)$ 
is equal to the number of isoclasses of simple $A$-modules of projective dimension $2i$. 
\end{lem}

For any  $i\geq 1$, the rank of $\varepsilon^i(A)$ is 
no more than the rank of $\varepsilon^{i-1}(A)$.
There is an integer $m$ such that  $\varepsilon^{m}(A)$ is selfinjective.
Then $\chi(\varepsilon^{m}{A})$ is equal to the number of isoclasses 
of simple $A$-modules of even projective dimension.
This gives an alternative proof of Madsen's criterion \cite{Mad05}, which states that
a Nakayama algebra $A$ has finite global dimension if and only if 
there is a simple $A$-module of even projective dimension.

Following \cite{MMG20}, the defect $\defect A$ of an algebra $A$ is 
the number of isoclasses of indecomposable injective non-projective $A$-modules. 
If $A$ is a Nakayama algebra, 
the defect of $A$ is equal to the number of red vertices in the resolution quiver of $A$. 
By \eqref{eq:delta} we have following.

\begin{lem} \label{lem:defect}
Let $A$ be a Nakayama algebra and $i\geq 0$. 
Then  $\defect\varepsilon^i(A)$ is equal to 
the number of isoclasses of simple $A$-modules of projective dimension $2i+1$. 
\end{lem}

The projective objects in $\scF$ are precisely 
the projective $A$-modules  contained in $\scF$.
The injective objects in $\scF$ are more complicated.
Let us give a description.

Let $S$ be a non-leaf vertex in $R(A)$. By \cite[A.6]{Rin21} 
there is a unique vertex $T$ such that $\tau^{-1}T$ is black and $\gamma(T)=S$.
Consider the exact sequence
\begin{equation} \label{eq:inj}
 0\to J(\tau^{-1}S) \to P(T) \overset{w}\to   P(I(\tau^{-1}T)) \to  I(\tau^{-1}T)\to 0,
\end{equation}
where $w\in A$ is the concatenation of the maximal nonzero path ending at $\tau^{-1}T$ 
and the arrow starting at $\tau^{-1}T$.

We have the following; compare \cite[3.18]{Sen19}.

\begin{lem} \label{lem:inj}
The assignment $S\mapsto J(\tau^{-1}S)$ gives a bijection between
the isoclasses of simple $A$-modules of injective dimension not equal to $1$ 
and the isoclasses of indecomposable injective objects in $\scF$.
It restricts to a bijection between
the isoclasses of simple $A$-modules of injective dimension $3$ 
and the isoclasses of indecomposable injective non-projective objects in $\scF$.
\end{lem}

\begin{proof}
If $T$ is a non-leaf vertex, then $I(\tau^{-1}T)$ is projective and $w$ is zero. 
It follows that $J(\tau^{-1}S)$ is an indecomposable injective $A$-module contained in $\scF$.

If $T$ is a leaf, then $I(\tau^{-1}T)$ is not projective and $w$ is nonzero.
Since $P(T)$ is an injective $A$-module, $J(\tau^{-1}S)$ is an indecomposable object in $\scF$.
Since the image of $w$ is a proper submodule of a base module,
we infer that $J(\tau^{-1}S)$ is an injective object in $\scF$.

If $S$ and $S'$ are not isomorphic, the socle modules of 
$J(\tau^{-1}S)$ and $J(\tau^{-1}S')$ are not isomorphic. 
Then $J(\tau^{-1}S)$ and $J(\tau^{-1}S')$ are  not isomorphic.
Note that the number of isoclasses of indecomposable injective objects in $\scF$ is equal to 
the number of non-leaf vertices in $R(A)$.
Then the assignment is a bijection.

By Lemma \ref{lem:leaf} $J(\tau^{-1}S)$ is a non-projective object in $\scF$ if and only if 
$\Delta(S)$ is a leaf in $R(\varepsilon{A})$. 
This happens if and only if the injective dimension of $S$ is $3$. 
\end{proof}

The translation $\tau$ induces a translation $\tau_\varepsilon$ one the simple objects in $\scF$.
For any vertex $\Delta(S)$ in $R(\varepsilon{A})$, by  \cite[B.4]{Rin21} we have
\[\gamma_\varepsilon(\Delta(S))=\tau_{\varepsilon}\soc_\scF P(\Delta(S))=\Delta(\gamma S).\]
The assignment $\Delta(S)\mapsto S$ yields an injective map of  resolution quivers
\[\varphi\colon R(\varepsilon{A})\to R(A).\]
We denote by $R(A)'$ the image of $\varphi$. 
Then $R(A)'$ is obtained from $R(A)$ by removing all leaf vertices.
Note that $\varphi$ preserves the weight.
However, $\varphi$ may not preserve the vertex colors.

A simple $A$-module $S$ is called $*$-black 
if $\tau^{-1}S$ is black, and $*$-red otherwise.  
For any resolution quiver $R$, denote by $\mathrm{Leaf}\-R,\mathrm{BLeaf}\-R$, and $\mathrm{*BLeaf}\-R$ 
the set of all leaf vertices, black leaf vertices, and $*$-black leaf vertices, respectively.

For each leaf $S$ in $R(A)'$, there is a unique black leaf $T_0$ in $R(A)$ 
such that  $\gamma(T_0)=S$, and a unique $*$-black leaf $T_1$ in $R(A)$ 
such that $\gamma(T_1)=S$.
There is an injective map 
\[\varphi_0\colon\mathrm{Leaf}\-R(\varepsilon{A}) \to \mathrm{BLeaf}\-R(A)\]
sending $\Delta(S)$ to $T_0$ and an injective map
\[\varphi_1\colon\mathrm{Leaf}\-R(\varepsilon{A}) \to \mathrm{*BLeaf}\-R(A)\]
sending $\Delta(S)$ to $T_1$.  
Consider the following two conditions.
\begin{enumerate}
  \item[(B)] For each black leaf in $R(A)$, its successor is a leaf in $R(A)'$.
  \item[(B$^*$)] For each $*$-black leaf in $R(A)$, its successor is a leaf in $R(A)'$.
\end{enumerate}
Then $A$ satisfies (B) if and only if the map $\varphi_0$ is bijective, 
and $A$ satisfies (B$^*$) if and only if the map $\varphi_1$ is bijective.

We have the following.

\begin{prop}
Let $A$ be a Nakayama algebra. 
The following are equivalent.
\begin{enumerate}
    \item $A$ satisfies (B);
    \item $\varphi$ preserves the vertex colors;
    \item the number of red vertices in $R(\varepsilon{A})$ and $R(A)'$ is equal;
    \item any minimal projective object in $\scF$ is a minimal projective $A$-module.
\end{enumerate}
\end{prop}

\begin{proof}
$(1) \iff (3)$
Recall that the number of red vertices is equal to the number of leaf vertices  in any resolution quiver.
We have
\[|\mathrm{Red}\-R(\varepsilon{A})|=|\mathrm{Leaf}\-R(\varepsilon{A})|\leq |\mathrm{BLeaf}\-R(A)|\]
and 
\[|\mathrm{BLeaf}\-R(A)|=|\mathrm{Leaf}\-R(A)|-|\mathrm{RLeaf}\-R(A)=|\mathrm{Red}\-R(A)'|.\]
The equality holds if and only if $\varphi_0$ is a bijection.

$(2)\iff (3)$
Let $P$ be an indecomposable projective object in $\scF$. 
If $P$ is minimal projective in $\mod A$, then $P$ is minimal projective in $\scF$.
This shows that $\varphi$  preserves red vertices.
Then $\varphi$ preserves the vertex colors if and only if 
the number of red vertices in $R(\varepsilon{A})$ and $R(A)'$ is equal.

$(2)\iff (4)$ This is obvious.
\end{proof}

Similarly, we have the following.

\begin{prop} \label{prop:cond-Bstar}
Let $A$ be a Nakayama algebra. The following are equivalent.
\begin{enumerate}
    \item $A$ satisfies (B$^*$);
    \item $\varphi$ preserves the vertex $*$-colors;
    \item the number of $*$-red vertices in $R(\varepsilon{A})$ and $R(A)'$ is equal;
    \item any injective  projective object in $\scF$ is an injective  projective $A$-module.
\end{enumerate}
\end{prop}

\begin{cor} \label{cor:cond-B}
Let $A$ be a Nakayama algebra satisfying (B). 
Then $A$ is Gorenstein if and only if $\varepsilon(A)$ is Gorenstein.
Furthermore, we have
\begin{equation} \label{eq:injdim}
\mathrm{id}_A\,A=\mathrm{id}_{\varepsilon(A)}\,\varepsilon(A)+2
\end{equation}
provided that $A$ is not selfinjective.
\end{cor}

\begin{proof}
Since $A$ satisfies (B), then $\varphi$ preserves the vertex colors. 
Recall that  $A$ has finite global dimension if and only if 
$\varepsilon(A)$ has  finite global dimension. 
Then by  Lemma \ref{lem:finite-vd} the algebra $A$ is Gorenstein if and only if 
$\varepsilon(A)$ is Gorenstein.

If $\mathrm{id}_A\,A = 1$, then $A$ is hereditary by Lemma \ref{lem:finite-vd}.
There is a black leaf in $R(A)$. This is impossible.
The equation \eqref{eq:injdim} follows from \cite[3.23]{Sen19}.
\end{proof}

\begin{exm}
Let $A$ be a Nakayama algebra with Kupisch series $(3,2,2)$.
The resolution quiver $R(A)$ is the following.
One checks that $A$ satisfies (B).

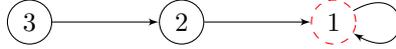
\[
\begin{tikzpicture}
% vertices
\node[red-vertex] (1) at  (0,0) {1};
\node[black-vertex] (2) at  (-2,0) {2};
\node[black-vertex] (3) at  (-4,0) {3};
%edges
\draw[edge] (3) to (2);
\draw[edge] (2) to (1);
\draw[edge] (1) to[out=30, in=-30, looseness=8] (1);
\end{tikzpicture}
\]

Note that $\Delta(1)=S(1)$ and $\Delta(2)=P(2)$.
The Kupisch series of $\varepsilon(A)$ is $(2,1)$.
The resolution quiver of $\varepsilon(A)$ is the following.

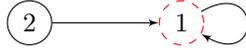
\[
\begin{tikzpicture}
% vertices
\node[red-vertex] (1) at  (0,0) {1};
\node[black-vertex] (2) at  (-2,0) {2};
%edges
\draw[edge] (2) to (1);
\draw[edge] (1) to[out=30, in=-30, looseness=8] (1);
\end{tikzpicture}\]
The colors of $1$ and $2$ are preserved.
\end{exm}

\begin{exm}
Let $A$ be a Nakayama algebra with Kupisch series $(3,2,1)$.
The resolution quiver $R(A)$ is the following.
One sees that $A$ does not satisfy (B).

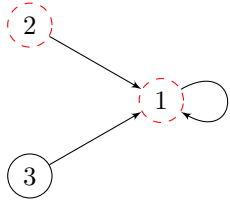
\[\begin{tikzpicture}
% vertices
\node[red-vertex] (1) at  (0,0) {1};
\node[red-vertex] (2) at  (-2,0) {2};
\node[black-vertex] (3) at  (2,0) {3};
%edges
\draw[edge] (2) to (1);
\draw[edge] (3) to (1);
\draw[edge] (1) to[out=120, in=60, looseness=8] (1);
\end{tikzpicture}\]

We have $\Delta(1)=P(1)$.
The Kupisch series of $\varepsilon(A)$ is $(1)$.
The resolution quiver of $\varepsilon(A)$ is the following.

\[
\begin{tikzpicture}
% vertices
\node[black-vertex] (1) at  (0,0) {1};
%edges
\draw[edge] (1) to[out=120, in=60, looseness=8] (1);
\end{tikzpicture}\]
The color of $1$ is not preserved.
\end{exm}

\begin{exm}
Let $A$ be a Nakayama algebra with Kupisch series $(8,7,7,6)$.
The resolution quiver of $R(A)$ is the following.
It is easy to see that $A$ satisfies (B$^*$).

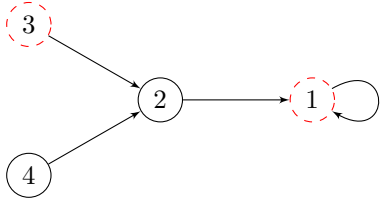
\[
\begin{tikzpicture}
% vertices
\node[red-vertex] (1) at  (2,0) {1};
\node[black-vertex] (2) at  (0,0) {2};
\node[red-vertex] (3) at  (-1.73,1) {3};
\node[black-vertex] (4) at  (-1.73,-1) {4};
%edges
\draw[edge] (1) to[out=30, in=-30, looseness=8] (1);
\draw[edge] (2) to (1);
\draw[edge] (3) to (2);
\draw[edge] (4) to (2);
\end{tikzpicture}
\]

We have two exact sequences
\[\begin{aligned}
    0\to J(4) \to P(1) \to P(1) \to I(4)\to 0,\\
    0\to J(1) \to P(3) \to P(1) \to I(2)\to 0.
\end{aligned}\]
Note that $J(4)$ is an injective projective $A$-module, 
and $J(1)$ is a non-projective $A$-module.
\end{exm}

\section{Proof of the main result}
In this section, we study the dominant dimension of Nakayama algebras 
by means of resolution quivers.
Then we prove the main result of the present paper.

Let $A$ be an  algebra and $M$ be an $A$-module. 
Take a minimal injective resolution
\[0\to M\to I^0\to I^1 \to \cdots\to I^n \to \cdots\]
of $M$. 
The  dominant dimension $\mathrm{dom}.\dim_AM$ of $M$ is 
the maximal integer $n$ or $\infty$ such that $I^j$ is projective for all $j<n$.
Dually, take a minimal projective resolution
\[\cdots \to P_n \to \cdots P_1 \to P_0 \to M\to 0\]
of $M$. 
The codominant dimension $\mathrm{codom}.\dim_AM$ of $M$ is  
the maximal integer $n$ or $\infty$ such that $P_j$ is injective for all $j<n$.
The left and right dominant dimensions of $A$ are  equal; see \cite{Mul68}.
It is known  that  the dominant dimension of any Nakayama algebra is at least $1$.

Following \cite{Iya07,IS18}, for an integer $d\geq 1$, an algebra $A$ is called higher Auslander if
\[\mathrm{dom}.\dim A = d = \mathrm{gl}.\dim A,\] 
and $A$ is called  minimal Auslander-Gorenstein if 
\[\mathrm{dom}.\dim A = d = \mathrm{id}_A\,A.\] 
Higher Auslander algebras are just the minimal Auslander-Gorenstein algebras of finite global dimension.

For algebras of selfinjective dimension $1$, we have the following.

\begin{prop} \label{prop:ausgor1}
    Let $A$ be a Nakayama  algebra and $R(A)$ its resolution quiver. 
    Then the selfinjective dimension of $A$ is equal to $1$ if and only if 
    \begin{enumerate}
        \item $\dist A=1$;
        \item $R(A)$ is connected of weight $1$;
        \item all cyclic vertices in $R(A)$ are $*$-black.
    \end{enumerate}
\end{prop}

\begin{proof}
    $``\implies"$   
    Since $\mathrm{id}_A\,A=1$, 
    by Lemma \ref{lem:finite-vd} we have $\mathrm{gl}.\dim A=1$.
    There is a direct sum decomposition of algebras
    \[A=A_1\oplus A_2\oplus \cdots\oplus A_r,\] 
    where $A_i$ is a hereditary linear Nakayama algebra for each $i\leq r$.
    By Lemma \ref{lem:finite-gld} $\dist A=1$ and $R(A)$ is connected of weight  $1$.  
    Let $S$ be a cyclic vertex in $R(A)$.
    Then $S$ is the simple module at the source of $A_i$ 
    and $\tau^{-1}{S}$  is the simple module at the sink of  $A_{i-1}$. 
    Since $\tau^{-1}S$ is projective, it is black.

    $``\impliedby"$
    Let $T$ be a black vertex and $S$ its successor.
    Since $\dist A=1$,  we have $S$ is cyclic.
    Then $S$ is $*$-black. By Lemma \ref{lem:black} $P(S)$ is injective.
    By \eqref{eq:elementary} there is an exact sequence 
    \[0\to \Delta(S)\to P(\tau T) \overset{\alpha}\to P(T) \to T \to 0.\]
    Since $R(A)$ is connected of weight $1$, 
    by Lemma \ref{lem:finite-gld} the global dimension of $A$ is at most $2$.
    It follows from \eqref{eq:delta} that  $\Delta(S)=P(S)$.
    Then $\alpha$ is zero and $T$ is projective.
    We conclude that the selfinjective dimension of $A$ is equal to $1$.
\end{proof}

\begin{rmk} \label{rmk:hereditary}
Let $A$ be a Nakayama algebra of global dimension $1$.
Since not all connected component of $A$ are simple,
there exists at least one black leaf in $R(A)$.
\end{rmk}

\begin{exm}
Let $A$ be a Nakayama algebra with Kupisch series $(1,3,2,1)$.
Then $A$ is the direct sum of two  hereditary linear Nakayama algebras.
The resolution quiver of $R(A)$ is the following.

\[
\begin{tikzpicture}
% vertices
\node[black-vertex] (1) at  (0,0) {1};
\node[red-vertex] (2) at  (2,0) {2};
\node[red-vertex] (3) at  (-1.73,1) {3};
\node[black-vertex] (4) at  (-1.73,-1) {4};
%edges
\draw[edge] (1) to[bend left] (2);
\draw[edge] (2) to[bend left] (1);
\draw[edge] (3) to (1);
\draw[edge] (4) to (1);
\end{tikzpicture}\]

Observe that  $4$ is a black leaf in the resolution quiver of $A$.
\end{exm} 

For  algebras of dominant dimension at least $2$, we need the following.

\begin{lem} \label{lem:domdim2}
Let $A$ be a Nakayama  algebra. 
The following are equivalent.
\begin{enumerate}
    \item $\mathrm{dom}.\dim A\geq 2$;
    \item every leaf in $R(A)$ is $*$-black;
    \item every $*$-red vertex  in $R(A)$ is the successor of a $*$-black vertex.
\end{enumerate}
\end{lem}

\begin{proof}
$(1) \implies (2)$
Let $S$ be a $*$-red vertex in $R(A)$.
Take a $*$-black vertex $T$ such that $\gamma(T)=\gamma(S)$.
Then $P(T)$ is injective.
There is an exact sequence
\[ 0\to P(S) \to P(T) \to I(\tau^{-1}S).\]
Since $\mathrm{dom}.\dim A\geq 2$, we have $I(\tau^{-1}S)$ is projective.
Then $S$ is not a leaf.

$(2) \implies (3)$
Let $S$ be a $*$-red vertex in $R(A)$.
Since $S$ is not a leaf, there is a $*$-black vertex $T$ such that $\gamma(T)=S$.
Then $S$ is the successor of a $*$-black vertex.

$(3) \implies (1)$ 
This follows from \cite[6.2]{NRTZ19}.
Let $P(T)$ be an indecomposable non-injective projective $A$-module.
Then $T$ is $*$-red and there is a $*$-black simple $A$-module $S$ such that $\gamma(S)=T$.
Since $S$ is $*$-black, $P(S)$ and $I(\tau^{-1}T)$ are isomorphic.
Then we have $\mathrm{dom}.\dim A\geq 2$.
\end{proof}

\begin{prop} \label{prop:ausgor2}
    Let $A$ be a Nakayama  algebra and $R(A)$ its resolution quiver. 
    Then $A$ is minimal Auslander-Gorenstein of selfinjective dimension $2$ if and only if 
    \begin{enumerate}
        \item $\dist A=1$;
        \item  $|\gamma^{-1}(S)|\leq 2$ for every cyclic vertex $S$ in $R(A)$;
        \item all cyclic vertices in $R(A)$ are  black.
    \end{enumerate}
\end{prop}

\begin{proof}
    ``$\implies$'' 
    (1) If $A$ has finite global dimension, then $\dist A=1$ by Lemma \ref{lem:finite-gld}.
    If $A$ has infinite global dimension, by Lemma \ref{lem:finite-vd} we have $\dist A=1$.

    (2) Suppose three distinct vertices $T,\tau T,\tau^2T$ have the same successor in $R(A)$.
    Then $T$ and $\tau T$ are red.
    Since $\mathrm{dom}.\dim A\geq 2$, 
    by Lemma \ref{lem:domdim2} $\tau T$ and $\tau^2T$ are non-leaf vertices.
    Since $\dist A=1$, both $\tau T$ and $\tau^2T$  are cyclic.
    This is impossible.

    (3) Let  $S$ be a red vertex in $R(A)$. 
    By Lemma \ref{lem:domdim2} $\tau^{-1}S$ is black.
    Then $P(S)$ is injective.
    Since $\mathrm{id}_A\,A=2$, there is an exact sequence
    \[  0\to  P(\tau S) \to P(S) \to I(S) \to I(\tau^{-1} S)\to 0.\]
    By Lemma \ref{lem:leaf} $S$ is a leaf in $R(A)$.
    In particular, $S$ is not cyclic.

    ``$\impliedby$'' 
    Since all cyclic vertices in $R(A)$ are black,
    by Lemma \ref{lem:finite-vd}  the algebra $A$ is  Gorenstein.
    Since $\dist A=1$, the selfinjective dimension of $A$ is at most $2$.
    By Remark \ref{rmk:hereditary}  $A$ is not hereditary. 
    Then the selfinjective dimension of $A$ is $2$.

    Let $S$ be a leaf vertex in $R(A)$. 
    Since $\dist A=1$ and all cyclic vertices are black, 
    then $S$ is red and $\gamma(S)=\gamma(\tau S)$. 
    Since $\gamma(S)$ has at most $2$ predecessors, $\tau^{-1}S$ is black.
    By Lemma \ref{lem:domdim2}, the dominant dimension of $A$ is at least $2$.
    Then $A$ is minimal Auslander-Gorenstein of selfinjective dimension $2$.
\end{proof}

\begin{exm}
Let $A$ be a Nakayama algebra with Kupisch series $(6,7,6,7)$.
The resolution quiver $R(A)$ is the following.

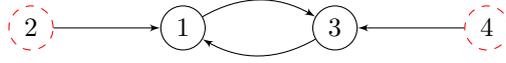
\[\begin{tikzpicture}
% vertices
\node[black-vertex] (1) at  (-1,0) {1};
\node[red-vertex] (2) at  (-3,0) {2};
\node[black-vertex] (3) at  (1,0) {3};
\node[red-vertex] (4) at  (3,0) {4};
%edges
\draw[edge] (1) to[bend left] (3);
\draw[edge] (3) to[bend left] (1);
\draw[edge] (2) to (1);
\draw[edge] (4) to (3);
\end{tikzpicture}\]

It is easy to check that $A$ fulfills the conditions in Proposition \ref{prop:ausgor2}. 
Then $A$ is a minimal Auslander-Gorenstein algebra of selfinjective dimension $2$.
\end{exm}

For algebras of dominant dimension at least $3$, recall the following key result from \cite[4]{STZ24}.
We give a short proof here.

\begin{lem} \label{lem:domdim3}
Let $A$ be a Nakayama algebra. 
The following are equivalent.
\begin{enumerate}
   \item $\mathrm{dom}.\dim A\geq 3$;
   \item every leaf  in $R(A)$ is black and $*$-black;
   \item $\gamma^{-1}(\gamma S)=\{S\}$ for every leaf $S$ in $R(A)$;
   \item $\defect A=\defect \varepsilon(A)$;
   \item $\mathrm{dom}.\dim A=\mathrm{dom}.\dim \varepsilon(A)+2$.
\end{enumerate}
\end{lem}

\begin{proof}
$(1) \implies (2)$
Let $S$ be a leaf in $R(A)$.
Since $\mathrm{dom}.\dim A\geq 2$,  
by Lemma \ref{lem:domdim2}  $S$ is $*$-black and $P(S)$ is injective.
Assume $S$ is red. There is an  exact sequence
\[0 \to P(\tau S) \to P(S) \to I(S)\to I(\tau^{-1}S)\to 0.\]
Since $\mathrm{dom}.\dim A\geq 3$, then $I(\tau^{-1}S)$ is projective.
This is impossible. 
We conclude that $S$ is black and $*$-black.

$(2) \iff (3)$
This follows from Corollary \ref{cor:two-black}.

$(3) \iff (4)$
Note that $\defect\varepsilon(A)$ is equal to the number of leaf vertices in $R(A)'$.
Then $\defect A$ and $\defect \varepsilon(A)$ are equal
if and only if every leaf  in $R(A)$ is the unique predecessor of its successor.

$(2) \implies (5)$
Let $T$ be a leaf in $R(A)$ and $S$  the successor of $T$.
Since $\tau^{-1}T$ is  black, 
by \eqref{eq:inj} there is an exact sequence 
\[ 0\to  J(\tau^{-1}S) \to P(T) \overset{w}\to   P(I(\tau^{-1}T))\to  I(\tau^{-1}T)\to 0.\]

Since $T$ is a leaf, we have $w\neq 0$.
Note that $A$ satisfies (B$^*$). 
By Proposition \ref{prop:cond-Bstar} any projective injective object in $\scF$ 
is a projective injective $A$-module. Then
\[\mathrm{codom}.\dim_A I(\tau^{-1}T)=\mathrm{codom}.\dim_{\scF} J(\tau^{-1}S)+2.\]
By Lemma \ref{lem:inj} we have $\mathrm{dom}.\dim A=\mathrm{dom}.\dim \varepsilon(A)+2$.

$(5) \implies (1)$ 
Since $\mathrm{dom}.\dim \varepsilon(A)\geq 1$, we have $\mathrm{dom}.\dim A\geq 3$.
\end{proof}

Note that $\mathrm{dom}.\dim A\geq 3$ implies (B) and (B$^*$). 
We have the following.

\begin{cor} \label{cor:color-preserve}
Let $A$ be a Nakayama  algebra of dominant dimension at least $3$ 
and  $\Delta(S)$  a vertex in $R(\varepsilon{A})$. Then 
	\begin{enumerate}
		\item $\Delta(S)$ is black if and only if $S$ is black;
		\item $\Delta(S)$ is $*$-black if and only if $S$ is $*$-black. 
	\end{enumerate}      
\end{cor}

Combining Lemma \ref{lem:domdim2} and \ref{lem:domdim3} with the previous result, we have the following.

\begin{cor} \label{cor:higherdomdim}
Let $A$ be a Nakayama algebra and $m\geq 0$.
\begin{enumerate}
 \item The following are equivalent.
   \begin{enumerate}
    \item  $\mathrm{dom}.\dim A\geq 2m +1$;
    \item  $S,\gamma(S),\cdots, \gamma^{m-1}(S)$ are black and $*$-black for every leaf $S$;
    \item  $|\gamma^{-1}\gamma^i(S)|=1$ for every leaf $S$ and $i\leq m$.
\end{enumerate}
 \item The following are equivalent.
 \begin{enumerate}
  \item $\mathrm{dom}.\dim A\geq 2m+2$;
  \item  $S,\gamma(S),\cdots, \gamma^{m-1}(S)$ are black and $*$-black, and $\gamma^{m}(S)$ is $*$-black 
  for every leaf $S$;
  \item $|\gamma^{-1}\gamma^m(S)|=1$ and $\gamma^m(S)$ is $*$-black for every leaf $S$ and $i\leq m$.
   \end{enumerate}
\end{enumerate}
\end{cor}

We have the following; compare \cite[4.4.2]{STZ24}.

\begin{prop} \label{prop:reduction}
    Let $A$ be a Nakayama algebra of dominant dimension  $d\geq 3$. 
    Then  $A$ is minimal Auslander-Gorenstein of selfinjective dimension $d$ if and only if 
    $\varepsilon(A)$ is minimal Auslander-Gorenstein of selfinjective dimension $d-2$.
\end{prop}

\begin{proof}
This follows from   Lemma \ref{lem:domdim3} and Corollary \ref{cor:cond-B}.
\end{proof}

Now we restate and prove Theorem \ref{mainthm:A}.

\begin{thm} \label{thm:odd}
Let  $A$ be a Nakayama  algebra and $m\geq 1$. 
Then $A$ is minimal Auslander-Gorenstein of selfinjective dimension $2m-1$ if and only if   
      \begin{enumerate}
        \item $\dist S=\dist A=m$ for every leaf $S$ in $R(A)$;
        \item $|\gamma^{-1}(S)|\leq 1$ for every noncyclic vertex $S$ in $R(A)$;
        \item $R(A)$ is connected of weight $1$; 
        \item all cyclic vertices in $R(A)$ are $*$-black.
    \end{enumerate}
\end{thm}

\begin{proof}
$``\implies"$
We prove by induction on $m$.
If $m=1$, this is done in Proposition \ref{prop:ausgor1}.
Suppose $m\geq 2$. By Proposition \ref{prop:reduction} the syzygy filtered algebra 
$\varepsilon(A)$ is minimal Auslander-Gorenstein of selfinjective dimension $2m-3$.
By induction hypothesis, the following statements hold.

\begin{enumerate}
    \item[(1a)] $\dist \Delta=\dist \varepsilon(A)=m-1$ for every leaf $\Delta$ in 
    $R(\varepsilon A)$;
    \item[(2a)] $|\gamma_\varepsilon^{-1}(\Delta)|\leq 1$  for every noncyclic vertex $\Delta$ in $R(\varepsilon A)$;
    \item[(3a)] $R(\varepsilon A)$ is connected of weight $1$;
    \item[(4a)] all cyclic vertices in $R(\varepsilon A)$ are $*$-black.
\end{enumerate}

Since there is  an injective weight preserving map  of resolution quivers 
from $R(\varepsilon{A})$ to $R(A)$, then (1), (2) and (3) hold.
By Corollary \ref{cor:color-preserve} the statement (4) holds.

$``\impliedby"$
We prove by induction on $m$.
If $m=1$, this is shown in Proposition \ref{prop:ausgor1}.
Suppose $m\geq 2$ and let $S$ be a  leaf in $R(A)$.
Since $\dist S\geq 2$, the successor of $S$ is not cyclic.
Then $S$ is the unique predecessor of its successor. 
By Lemma \ref{lem:domdim3} the dominant dimension of $A$ is at least $3$.

Since there is  an injective weight preserving map  of resolution quivers 
from $R(\varepsilon{A})$ to $R(A)$, then (1a), (2a) and (3a) hold.
By Corollary \ref{cor:color-preserve} the statement (4a) holds.

By the induction hypothesis, the algebra $\varepsilon(A)$ is 
minimal Auslander-Gorenstein of selfinjective dimension $2m-3$. 
Since the dominant dimension of $A$ is at least $3$, 
by Proposition \ref{prop:reduction} 
the algebra $A$ is minimal Auslander-Gorenstein of selfinjective dimension $2m-1$.
\end{proof}

Let us restate Theorem \ref{mainthm:B} as follows.  The proof is similar, we omit it here.

\begin{thm} \label{thm:even}
    Let $A$ be a Nakayama  algebra and $m\geq 1$. 
    Then $A$ is minimal Auslander-Gorenstein of selfinjective dimension $2m$  if and only if 
    \begin{enumerate}
        \item $\dist S=\dist A=m$ for every leaf $S$ in $R(A)$;
        \item $|\gamma^{-1}(S)|\leq 1$ for every noncyclic vertex $S$ in $R(A)$;
        \item $|\gamma^{-1}(S)|\leq 2$ for every cyclic vertex $S$ in $R(A)$;
        \item all cyclic vertices in $R(A)$ are black.
    \end{enumerate}
\end{thm}

For a Nakayama algebra $A$, denote by $\cyc A$ the number of cyclic vertices in its resolution quiver.

\begin{cor}
    Let $A$ be a Nakayama algebra.
    If $A$ is a minimal Auslander-Gorenstein algebra of even selfinjective dimension, then 
    $\cyc A \leq \defect A$.
\end{cor}

\section{Applications and examples}
In this section, we present some applications and examples of our results.

Let $A$ be a Nakayama algebra  with Kupisch series
\[(c_1,c_2,\cdots,c_{n}).\]
Denote by  $\bar{A}$ the Nakayama algebra with Kupisch series
\[(c_1+n,c_2+n,\cdots,c_{n}+n).\]
It is direct to show that there is a color preserving isomorphism between $R(\bar{A})$ and $R(A)$.
However, the weight of $R(\bar{A})$ is grater than the weight of $R(A)$.

By Theorem \ref{thm:even}, we have the following.

\begin{prop}
Let $A$ be a Nakayama algebra and $d$ an even number.
Then $A$ is minimal Auslander-Gorenstein of selfinjective dimension $d$ 
if and only if so is $\bar{A}$.
\end{prop}

Let $2\leq m\leq n$. 
Denote by $\bbA_n$ the hereditary linear Nakayama algebra of rank $n$.
Let  $N_n(m)$ be the quotient algebra of $\bbA_n$ 
via the ideal generated by all paths of length $m$.
The algebra $N_n(m)$ is a linear Nakayama algebra and has finite global dimension.
The Kupisch series of $N_n(m)$ is 
\[((m)^{n-m+1},m-1,\cdots,2,1).\]

If $m=2$, then $N_n(m)$ is higher Auslander of global dimension $n-1$.
Suppose  $m>2$ and $N_n(m)$ is a higher Auslander algebra. 
We write 
\[n=qm+r\] 
such that $q,r$ are integers and $0 \leq r\leq m-1$.
Since $S(1)$ is cyclic and $S(1)$ has $m$ predecessors in the resolution quiver, 
the global dimension of $N_n(m)$ is odd by Theorem \ref{thm:even}.
Since $S(n-r+1)$ is cyclic, $S(n-r)$ is black by Theorem \ref{thm:odd}. 
It follows  that $r=0$ and $m$ divides $n$.

If $m$ divides $n$, it is routine to  check that $N_n(m)$ is  higher Auslander 
of global dimension $2n/m-1$.

\begin{prop}
The algebra  $N_n(m)$ is a higher Auslander algebra if and only if $m=2$ or $m$ divides $n$.
In both cases, the global dimension of $N_n(m)$ is $2n/m-1$.
\end{prop}

For any Nakayama algebra $A$, by \cite{STZ24}
there is a unique Nakayama algebra $\varepsilon^{-1}{A}$ 
of dominant dimension $\geq 3$ 
such that $\varepsilon(\varepsilon^{-1}{A})$ is naturally isomorphic to $A$.
The algebra $\varepsilon^{-1}A$ is given as follows.
Suppose the Kupisch series of $A$ is   
\[c(A)=(c_1,c_2,\cdots,c_{n}). \] 
Denote by $d_a$ the maximum of $c_{a}-c_{a-1}$ and $0$ for each $a\leq n$,
and by $d$ the sum of all $d_a$. 
The subscript is taken modulo $n$.
Let $\tilde{c}_a$ be a series of length $1+d_{a+1}$ with
\[(\tilde{c}_a)_\ell=c_a+\sum_{i=1}^{c_a+\ell-1}d_{a+i},\;\;1\leq \ell\leq 1+d_{a+1}.\]
Then $\varepsilon^{-1}{A}$ is a Nakayama algebra of rank $n+d$ with Kupisch series 
\[c(\varepsilon^{-1}A)=(\tilde{c}_1,\tilde{c}_2,\cdots,\tilde{c}_{n}).\] 
Note that $\chi(\varepsilon^{-1}A)$ equals the number of simple components of $A$.

Recall the following result; see \cite{Sen20,STZ24}.

\begin{prop} \label{prop:aus-cyclic}
Let $\Lambda$ be a cyclic Nakayama algebra and $d\geq 3$.
Then $\Lambda$ is higher Auslander of global dimension $d$ if and only if  $\Lambda$ is isomorphic to
\[\varepsilon^{-1}(A)\]
where $A$ is a cyclic Nakayama algebra which is higher Auslander of global dimension $d-2$ or isomorphic to
\[\varepsilon^{-1}(A_1\oplus A_2\oplus\cdots\oplus A_r)\]
where $A_i$ is a  linear Nakayama algebra which is higher Auslander of global dimension $d-2$ 
for each $i\leq r$.
\end{prop}

By Proposition \ref{prop:reduction}, we have the following.

\begin{prop}\label{prop:aus-linear}
Let $\Lambda$ be a linear Nakayama algebra and $d\geq 3$.
Then $\Lambda$ is higher Auslander of global dimension $d$ if and only if  $\Lambda$ is isomorphic to
\[\varepsilon^{-1}(\bbA_1\oplus A_1\oplus\cdots\oplus A_r)\]
where $A_i$ is a linear Nakayama algebra which is higher Auslander of global dimension $d-2$ for 
each $i\leq r$.
\end{prop}

\begin{exm}
Let $A$ be a Nakayama algebra with Kupisch series 
\[(1,3,2,1,2,1).\]
Then $\varepsilon^{-1}(A)$  is higher Auslander of global dimension $3$ with Kupisch series
\[(3,3,3,4,3,2,2,2,1).\]
\end{exm}

Nakayama algebras of global dimension $1$ are hereditary, 
and Nakayama algebras which are higher Auslander of global dimension  $2$ are  
the Auslander algebras of Nakayama algebras with radical square zero.
For any $d\geq 3$,  we can repeatedly apply Proposition \ref{prop:aus-linear} and \ref{prop:aus-cyclic} 
to get any connected Nakayama algebra which is higher Auslander of global dimension $d$.

\begin{exm}
Let $n\geq m\geq 0$ and $A$ be a non-simple linear Nakayama algebra. 
The  algebra
\[\Lambda=\varepsilon^{-n}(\bbA_1^m\oplus A)\]
is a connected Nakayama algebra; it is linear if and only if $n=m$. 
We have 
\begin{enumerate}
    \item $\varepsilon^{-m}(\bbA_1^m\oplus\bbA_n)\cong N_{mn+n}(n)$ for any $m\geq 1$ and $n\geq 2$;
    \item $\varepsilon^{-m}(\bbA_1^m\oplus N_{3}(2))\cong N_{2m+3}(2)$ for any $m\geq 1$.
\end{enumerate}
\end{exm}


\begin{thebibliography}{10}


\bibitem{ARS95}
M. Auslander, I. Reiten, S.O. Smal{\o}
{\em Representation theory of Artin algebras}. 
Cambridge Studies in Advanced Mathematics, 36. 
Cambridge University Press, Cambridge (1995).


\bibitem{Gus85}
W.H. Gustafson,
{\em Global dimension in serial rings}.
J. Algebra 97 (1985), 14-16.


\bibitem{HI20}
E.J. Hanson, K. Igusa, 
{\em Resolution quiver and cyclic homology criteria for Nakayama algebras}.
J. Algebra 553 (2020), 138-153.

\bibitem{Iya07} 
O. Iyama, 
{\em Higher-dimensional Auslander-Reiten theory on maximal orthogonal subcategories}.
Adv. Math. 210 (2007), 22-50. 


\bibitem{IS18} 
O. Iyama, {\O}. Solberg, 
{\em Auslander-Gorenstein algebras and precluster tilting}.
Adv. Math. 326 (2018), 200-240.

\bibitem{Mad05}
D. Madsen, {\em Projective dimensions and Nakayama algebras}.
Fields Institute Communications. 45. Amer. Math. Soc., Providence, RI, 2005. 247-265.


\bibitem{MMG20}
D. Madsen, R. Marczinzik, G. Zaimi,
{\em On the classification of higher Auslander algebras for Nakayama algebras}.
J. Algebra 556 (2020), 776-805.

\bibitem{Mul68}
B.J. M{\"u}ller,
{\em The classification of algebras by dominant dimension}.
Canadian J. Math. 20 (1968), 398-409.


\bibitem{NRTZ19}
V.C. Nguyen, I. Reiten, G. Todorov, S. Zhu,
{\em Dominant dimension and tilting modules}.
Math. Z. 292 (2019), 947-973.


\bibitem{Rin13}
C.M. Ringel, 
{\em The Gorenstein projective modules for the Nakayama algebras. I}.
J. Algebra 385 (2013), 241-261.

\bibitem{Rin21}
C.M. Ringel, 
{\em The finitistic dimension of a Nakayama algebra}.
J. Algebra 576 (2021), 95-145.

\bibitem{Rin22}
C.M. Ringel,
{\em Linear Nakayama algebras which are higher Auslander algebras}.
Comm. Algebra 50 (2022), 4842-4881.

\bibitem{Sen19}
E. Sen, 
{\em Syzygy filtrations of cyclic Nakayama algebras}.
arXiv preprint 1903.04645 (2019).


\bibitem{Sen20}
E. Sen
{\em Nakayama Algebras which are Higher Auslander Algebras}.
arXiv preprint 2009.03383 (2020).



\bibitem{Sen21}
E. Sen, {\em The $\varphi$-dimension of cyclic Nakayama algebras}. 
Comm. Algebra 49 (2021), no. 6, 2278-2299.



\bibitem{STZ24}
E. Sen, G. Todorov, S. Zhu,
{\em Defect invariant Nakayama algebras}.
arXiv preprint 2406.002541 (2024).




\bibitem{She17}
D. Shen, 
{\em A note on homological properties of Nakayama algebras}.
Arch. Math. 108 (2017), 251-261.

\bibitem{Zak69}
A. Zaks,
{\em Injective dimension of semi-primary rings}.
J. Algebra 13 (1969), 73-86.


\end{thebibliography}
\end{document}